\documentclass[reqno,12pt]{article}

\usepackage{amsmath,amsthm,amsfonts,amscd}
\usepackage{times}
\usepackage{verbatim}
\usepackage{afterpage}
\usepackage{xy}
\usepackage{graphicx}
\usepackage{enumerate}
\usepackage{diagrams}
\usepackage{amsmath,amsfonts,amssymb}
\usepackage{bbm,dsfont}
\usepackage{mathrsfs}
\usepackage{ifthen}
\usepackage{rotating}

\usepackage{makeidx}

\newtheorem{thm}{Theorem}[section]

\newtheorem{lem}[thm]{Lemma}
\newtheorem{prop}[thm]{Proposition}
\newtheorem{cor}[thm]{Corollary}

\theoremstyle{definition}

\newcommand{\C}{\mathbb{C}}
\renewcommand{\P}{\mathbb{P}}
\newcommand{\Z}{\mathbb{Z}}
\newcommand{\F}{\mathcal{F}}
\newcommand{\I}{\mathcal{I}}
\renewcommand{\O}{\mathcal{O}}
\newcommand{\x}{[x]}
\newcommand{\e}{\operatorname{e}}
\newcommand{\ch}{\operatorname{ch}_2}
\newcommand{\CH}{\operatorname{ch}}
\newcommand{\td}{\operatorname{td}}
\newcommand{\M}{M^L(Y,\beta)}

\newcommand{\supp}{\operatorname{supp}}
\newcommand{\Fh}{\hat{F}}
\newcommand{\Zh}{\hat{Z}}
\newcommand{\Ch}{\hat{C}}
\newcommand{\Gh}{\hat{G}}

\newcommand{\CM}{Cohen-Macaulay }
\newcommand{\im}{\operatorname{im}}
\renewcommand{\hom}{\operatorname{Hom_Y}}
\newcommand{\ext}{\operatorname{Ext}_Y}
\newcommand{\hilb}{\text{-Hilb}}
\renewcommand{\hat}{\widehat}
\renewcommand{\tilde}{\widetilde}

\begin{document}
\title{BPS invariants for resolutions of polyhedral singularities}
\author{Jim Bryan and Amin Gholampour}
\maketitle

\abstract{We study the BPS invariants of the preferred Calabi-Yau
resolution of ADE polyhedral singularities $\C^3/G$ given by
Nakamura's $G$-Hilbert schemes. Genus 0 BPS invariants are defined by
means of the moduli space of torsion sheaves as proposed by Sheldon Katz
\cite{Katz-BPS}. We show that these invariants are equal to half the
number of certain positive roots of an ADE root system associated to
$G$. This is in agreement with the prediction given in
\cite{BGh-Ghilb} via Gromov-Witten theory.}

\section{Introduction} \label{sec:Intro} The BPS invariants of a
Calabi-Yau threefold $Y$ were first defined by Gopakumar
and Vafa in the context of $M$-theory \cite{Go-Va}. These invariants
are conjecturally closely related to the Gromov-Witten invariants of
$Y$, and since their appearance, there has been much effort to put
them into a rigorous mathematical framework. The most satisfying
approaches so far have been proposed by Sheldon Katz \cite{Katz-BPS}
via studying the moduli space of torsion sheaves on $Y$, and by Rahul
Pandharipande and Richard Thomas \cite{PT-BPS} by means of moduli
space of stable pairs on $Y$. The first considers all curve classes
but is restricted to genus zero invariants, while the second can be
defined for all genera but is confined to irreducible curve
classes. In this paper we take the first point of view.

For a finite subgroup $G$ of $SO(3)$, let
\[
Y=G\hilb(\C^3)
\]
be Nakamura's Hilbert scheme of $G$-clusters in $\C^3$. By
\cite{Nak-Ghilb} $$\pi:Y\to \C^3/G$$ is a Calabi-Yau semi-small
resolution of singularities of $\C^3/G$, where $\pi$ is the
Hilbert-Chow morphism. In particular, the fiber of $\pi$ over the
origin is 1-dimensional. Moreover, it is proven in \cite{Nak-Ghilb}
that the reduced fiber over the origin, denoted by $F$, is a connected
configuration of rational curves, and there exists a natural bijection
between irreducible components of $F$ and non-trivial irreducible
representations of $G$. It is shown in \cite{BGh-Ghilb}, that $Y$ can
be realized as a surface fibration over $\C$. The central fiber,
$S_0$, is a surface with finite number of ordinary double point
singularities and $F \subset S_0$.

Let $\hat{G}$ be the pullback of $G$ under the natural double cover
map $SU(2)\to SO(3)$, and let
\[
S=\Gh\hilb(\C^2).
\]
$S$ is a smooth surface and is the resolution of the ADE type
singularity $\C^2/ \Gh$. McKay correspondence gives a natural
bijection between irreducible components of the exceptional curve on $S$,
which we denote by $\Fh$, and non-trivial irreducible representations
of $\Gh$. By \cite{Boissiere-Sarti} there exists a regular projective
morphism $f:S\to Y$ that factors through $S_0$ and it is the minimal
resolution of singularities of $S_0$. Moreover $f$ maps $\Fh$ onto $F$
in such a way that it maps isomorphically a component of $\Fh$
corresponding to a representation of $G$ and contracts a component of
$\Fh$ otherwise. In Figure \ref{fig:ADE} we show the ADE Dynkin
diagrams associated to $\Fh$ in each case. Black vertices stand for
the curves that are contracted by $f$.
\begin{figure} \label{fig:ADE}
\centering
\includegraphics{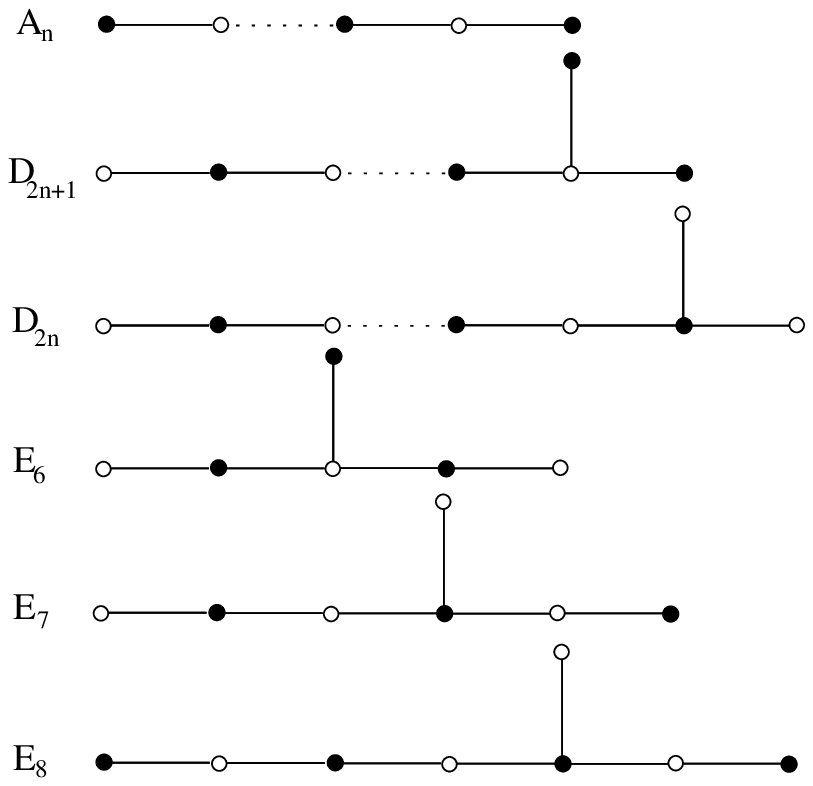}
\caption{ADE Dynkin diagrams}
\end{figure}

In order to define BPS invariants in this paper we make essential use
of the natural $\C^*$-action on $Y$ that is induced from the its
diagonal action on $\C^3$. This induces an action on $S$ with
respect to which $f$ is equivariant. Note that $F$, $\Fh$ and $S_0$
are invariant under these actions.

Let $L$ be a line bundle on $Y$ whose restriction to each irreducible
component of $F$ has a positive degree. Let $\beta \in H_2(Y,\Z)$ be a
curve class, and let $\M$ be the Simpson moduli space of $L$-stable
coherent sheaves $\F$ of pure dimension 1 on $Y$, with $\chi(\F)=1$
and $\ch(\F)=\beta$. Katz uses the moduli space $\M$ to define the
genus 0 BPS invariants of $Y$ in class $\beta$.
Since $Y$ is a Calabi-Yau threefold the Hilbert polynomial of any $\F$
of pure dimension 1 is
\[
P_\F(n)=(L\cdot \ch(\F))n+\chi(\F).
\]
By the condition $\chi(\F)=1$, semi-stability implies stability, hence
$\M$ admits a perfect obstruction theory by \cite{Thomas-Casson} and
hence a zero dimensional virtual cycle. $\M$ is not necessarily
compact because $Y$ is not, however it carries a $\C^*$-action that is
induced from the action on $Y$; the fixed locus of this action is
compact (see Theorem \ref{thm:fixed-locus}), so we can define BPS
invariants as equivariant residues of this virtual cycle at the fixed
locus.

Irreducible components of $F$ (respectively, of $\Fh$) represent a
basis for $H_2(Y,\Z)$ (respectively, for $H_2(S,\Z)$). The McKay
correspondence gives a natural identification of $H_2(S,\Z)$ with the
associated ADE root lattice denoted by $R$. This induces a map
$c:R^+\to H_2(Y,\Z)$ (via $f_*$) where $R^+\subset R$ is the set of
positive roots (see \cite{BGh-Ghilb} for more details). We say that
$\beta \in H_2(Y,\Z)$ corresponds to a positive root if it is in the
image of $c$.

One of the main result of this paper is determining the fixed locus of
the moduli space of BPS invariants:

\begin{thm}\label{thm:fixed-locus}
The $\C^*$-fixed point set of $\M$ consists of a single point if
$\beta$ corresponds to a positive root and is empty
otherwise. Moreover, the sheaf corresponding to the fixed point in
$\M$ is $\O_C$, the structure sheaf of a certain Cohen-Macaulay curve $C \subset Y$, which we define in
\S~\ref{sec:fixed-locus}.
\end{thm}
This theorem is proven in \S~\ref{sec:fixed-locus}. Computing
equivariant residues of the virtual class at the fixed point given by
Theorem \ref{thm:fixed-locus}, we can show:

\begin{thm}\label{thm:GVinvariants} Let $\beta \in H_2(Y,Z)$, and let $n^L_\beta(Y)$ denote the genus 0 BPS invariants of $Y$ in class $\beta$ (see \S~ \ref{sec:GVinvarinats}). Then
$$n^L_\beta(Y)=\begin{cases} \frac{1}{2}|c^{-1}(\beta)| &  \beta \; \text{corresponds to a positive root}\\0 & \text{otherwise}. \end{cases}$$
\end{thm}
This result is in agreement with the prediction of $n^L_\beta(Y)$ in
\cite{BGh-Ghilb} obtained by means of Gromov-Witten theory.
\begin{cor}
The genus 0 Gopakumar-Vafa conjecture \cite[Conj~2.3]{Katz-BPS} holds
for $Y$.
\end{cor}

Note that
Theorem \ref{thm:GVinvariants} shows that $n^L_\beta(Y)$ does not
depend on the choice of polarization $L$. Proof of Theorem
\ref{thm:GVinvariants} is given in \S~ \ref{sec:GVinvarinats}, in
the most elaborate case, i.e.  $Y=A_5\hilb(\C^3)$ and $\beta$
corresponds to the longest root in $E_8$ root system.

\section{Proof of Theorem \ref{thm:fixed-locus}}
\label{sec:fixed-locus} Let $F_1,\dots, F_r$ be the irreducible
components of $F$. We know each $F_i \cong \mathbb{P}^1$; let
$\alpha_i=\deg(L|_{F_i})$, note that $\alpha_i>0$ by the choice of
$L$.

\begin{lem}\label{lem: supp(F) is in S0}
Suppose $\F \in \M$ is $\C^*$-fixed, then $\supp(\F)$ is contained in
$S_0$.
\end{lem}

\begin{proof}
Since $F_1,\dots, F_r$ are the only 1 dimensional invariant
subvarieties of $Y$ \begin{equation}{\label{equ:reduced
support}}\supp(\F)_{\text{reduced}} \subseteq F \subset
S_0. \end{equation} $S_0 \subset Y$ is a Cartier divisor with trivial
normal bundle, hence there is an equivariant short exact
sequence $$0\to \O_Y \to \O_Y\to \O_{S_0}\to 0.$$ Tensoring with $\F$,
we get $$\F \to \F \to \F \otimes \O_{S_0}\to 0.$$ Since $\F$ is
stable the first map is either zero or an isomorphism. In the later
case we conclude $\F \otimes \O_{S_0}=0$ which is impossible by
(\ref{equ:reduced support}). Hence the first map in the exact sequence
above must be zero which implies $\F \cong \F \otimes \O_{ S_0}$, and
this proves the lemma.
\end{proof}

Since $\F$ is pure, $\supp(\F)$ is a Cohen-Macaulay curve. Note that
$\supp(\F)$ is uniquely determined by these further conditions (see
Lemma 3.2.i. in \cite{Katz-BPS} for a similar
situation): \begin{enumerate}[(i)] \item $\supp(\F)$ is a subscheme of
$S_0$, \item $\ch(\F)=\beta.$ \end{enumerate} The second condition
determines the multiplicity of $\supp(\F)$ at a generic point of each
$F_i$ (away from the singularities of $S_0$). This number is
well-defined and from now on we refer to it simply as multiplicity
along $F_i$.

\begin{lem}\label{lem: F=O_C}
Suppose $\F \in \M $ is $\C ^{*}$-fixed, then $\F \cong \O_C$ where $C$ is a Cohen-Macaulay curve.
\end{lem}

\begin{proof}
Let $C=\supp(\F)$. We know $C_{\text{reduced}}\subseteq F$ because $\F$ is $\C ^{*}$-fixed. Since $\F$ is 1 dimensional, $\chi(\F)=1$ implies that $\F$ has a
global section $s:\O_Y \to \F$. $\ker(s)$ is the ideal sheaf of a
subscheme $Z$ of $C$. We have $\O_Z=\operatorname{im}(s)$ is a
subsheaf of $\F$, hence by purity of $\F$, $Z$ has to be a
Cohen-Macaulay curve whose reduced support is a subset of $F$.  We
will show that $\O_Z$ will be a destabilizing subsheaf unless $Z=C$
and $\F \cong \O_C$.

To see this, let $m_i$ (respectively $m'_i$) be the multiplicity of
$C$ (respectively $Z$) along $F_i$.  Since $\O_Z$ is a subsheaf of
$\F$ we have $m'_i \le m_i$. By lemma \ref{lem:euler-char} proven
below, we know $\chi (\O_Z) \ge 1$, and by stability of $\F$ we must
have $$\frac{\chi(\O_Z)}{\sum m'_i \alpha_i} \le \frac{1}{\sum m_i
\alpha_i},$$ which is impossible unless $Z=C$ and $\F\cong \O_C$.
\end{proof}

\begin{lem}\label{lem: O_C with chi=1 is stable}
Suppose $C\subset Y$ is a proper, Cohen-Macaulay curve supported on $S_0$ with
$\chi(\O_C)=1$, then $\O _{C}$ is $\C ^{*}$-fixed and stable.
\end{lem}

\begin{proof}
Since $C$ is proper $C_{\text{reduced}}\subseteq F$ and hence $\O_C$ is $\C^*$-fixed. To prove the lemma we check the
stability condition for $\O_C$ (See \cite[Proposition
1.2.6.iv]{Huybrechts-Lehn-book}). Suppose $\I$ is a subsheaf of
$\O_C$. $\I$ determines a subscheme $Z$ of $C$. If $Z$ is zero
dimensional there is nothing to prove, so assume that $Z$ is one
dimensional.
Let $m_i$ be the multiplicity of $C$ along $F_i$.
Let $C'$ be a proper Weil divisor on $S_0$ having multiplicity $m'_i$ along $F_i$, where $m'_i$ is the generic multiplicity of $Z$ along $F_i$. Since $Z$ is a subscheme of $C$ we have $m'_i \le m_i$. By Lemma \ref{lem:euler-char} below, $\chi(\O_{C'}) \ge 1$, and hence $\chi(\O_Z) \ge 1$, because $Z$ differs from $C'$ at possibly a finite number of points which they add up to its Euler characteristic. Thus $$\frac{\chi(\O_Z)}{\sum m'_i \alpha_i}\ge \frac{1}{\sum m_i \alpha_i},$$
with the equality only if $Z=C$. This proves the stability of $\O_C$, and the claim follows.
\end{proof}

Theorem~\ref{thm:fixed-locus} follows from Lemmas \ref{lem: supp(F) is
in S0}, \ref{lem: F=O_C}, and \ref{lem: O_C with chi=1 is stable} so
to complete the proof it remains to prove the follow lemma which was
quoted above.

\begin{lem} \label{lem:euler-char}
Let $Z$ be a proper effective Weil divisor on $S_0$. Then
$\chi(\O_Z)\ge 1$ and equality holds if and only if the homology class of $Z$ corresponds to a positive root.
\end{lem}
\begin{proof}
 An effective Weil divisor $\Zh$ on $S$ which is supported on $\Fh$,
determines an element of $R^+$ via McKay correspondence. By the adjunction formula, we have
\[
\chi(\O_{\Zh})=-\Zh^2/2.
\]
Since $S$ is the resolution of rational singularities by
\cite[Proposition 1]{Artin} \begin{equation}
\label{equ:rational-singularity} \chi(\O_{\Zh})\ge 1, \end{equation}
and the equality holds if and only if $\Zh$ corresponds to a positive
root.

The idea of our proof is to relate $\chi(\O_Z)$ to $\chi(\O_{\Zh})$
where $\Zh \subset S$ is defined as follows.  Suppose $$Z=\sum_{i=1}^r
m_i F_i,$$ we define a divisor $$\Zh=\sum_{i=1}^r m_i \Fh_i +
\sum_{j=1}^s n_j E_j$$ on $S$ where $\Fh_i$ is the proper transform of
$F_i$ via $f:S \to S_0$, $$E_1,\dots,E_s$$ are exceptional curves of
$f$, and $n_j$ is given as follows. Let
\[
k_{j} = E_{j}\cdot \sum _{i} m_{i}\hat{F}_{i}
\]
be the sum of the multiplicities of the curves incident to
$E_{j}$. We define
\[
n_{j} = \left\lceil\frac{k_{j}}{2}\right\rceil.
\]
Note that
\[
\delta _{j} = \hat{Z}\cdot E_{j} = -2n_{j}+k_{j}
\]
is either 0 or $-1$ by construction.

Now we claim that $\chi(\O_Z)=\chi(\O_{\Zh})$. We show that $\im
(\hat{Z})$, the scheme theoretic image of $\hat{Z}$ under $f$, is
equal to $Z$. Since $f$ is an isomorphism away from singularities it
suffices to show that $\im(\Zh)$ does not carry any embedded points at
singularities. Let $Z'$ be a minimal divisor supported on $F$ such
that $C=Z+Z'$ is a Cartier divisor. By the choice of $\Zh$ one can see
that $f^*C=\Zh+\Zh'$, where $\Zh'$ is the proper transform of $Z'$. In
fact they are linearly equivalent divisors on $S$ and are isomorphic
away from exceptional curves. Since $f^*C$ is locally cut out by a
single equation it is obvious that
$C=\im(f^*C)$. But $$\im(\Zh)\subseteq \im(f^*C),$$ and $C$, being
locally cut out by a single equation on $S_0$ does not have any embedded
points. This shows that $\im(\Zh)$ does not have any embedded points either and
hence $\im(\Zh)=Z$.  Since $S$ is a resolution of rational
singularities $R^1f_* \O_S=0$, and moreover $R^2f_*\F=0$ for any
coherent sheaf $\F$ on $S$, so we get immediately $R^1f_*\O_{\Zh}=0$,
and therefore \begin{equation}\label{equ:spectral-sequence}
\chi(\O_{\Zh})=\chi(f_*\O_{\Zh})\end{equation} by Leray spectral
sequence.  Since $f_*\O_S\cong\O_{S_0}$ one can see the isomorphism of
ideal sheaves $$\I_Z=\I_{\im(\Zh)}\cong f_*\I_{\Zh}$$ and hence we
have a short exact sequence $$0\to \I_{Z} \to \O_{S_0} \to
f_*\O_{\Zh}\to 0.$$ Here we used the fact that $R^1f_*(\I_{\Zh})=0$
which is true because the restriction of $\I_{\Zh} \cong \O_{S}(-\Zh)$
on each $E_j$ is equal to $-\delta _{j}$ which is non-negative. The
short exact sequence above implies that $f_*\O_{\Zh}\cong \O_Z$. Now
the claim follows from (\ref{equ:spectral-sequence}).

To finish the proof of lemma we only need to show if the homology
class of $Z$ corresponds to a positive root then $\Zh$ is a positive
root. Since $n_j\geq 0$, $\hat{Z}$ is a positive root if and only if
$\Zh^2=-2$.

Let
\[
\tilde{Z} = \sum _{i}m_{i}\hat{F}_{i} + \sum _{j} \tilde{n}_{j}E_{j}
\]
be any positive root which corresponds to $Z$. Since $\tilde{Z}$ and $E_{j}$
are distinct positive roots in an ADE root system,
\[
\tilde{\delta }_{j} = \tilde{Z}\cdot E_{j}
\]
is equal to $-1$, $0$, or $1$. We compute
\begin{align*}
\hat{Z}^{2}-\tilde{Z}^{2}& = (\hat{Z}-\tilde{Z})\cdot
(\hat{Z}+\tilde{Z})\\
&=\sum _{j} (n_{j}-\tilde{n}_{j}) (\delta _{j}+\tilde{\delta }_{j}).
\end{align*}

Since
\begin{align*}
\delta _{j} = E_{j}\cdot \hat{Z}& = -2n_{j}+k_{j},\\
\tilde{\delta }_{j} = E_{j}\cdot \tilde{Z}&= -2\tilde{n}_{j}+k_{j}
\end{align*}
we get
\[
n_{j}-\tilde{n}_{j} = \frac{1}{2}\left(\tilde{\delta }_{j}-\delta _{j}
\right)
\]
and therefore
\[
\hat{Z}^{2} - \tilde{Z}^{2} = \sum _{j} \frac{1}{2}\left(\tilde{\delta }_{j}^{2} - \delta _{j}^{2} \right).
\]
Since $n_{j}-\tilde{n}_{j}$ is an integer, we know that $\delta _{j}$ is congruent to $\tilde{\delta }_{j}$ modulo 2. Then since $\delta _{j},\tilde{\delta }_{j} \in \left\{-1,0,1 \right\}$, we conclude that
\[
\tilde{\delta }^{2}_{j}-\delta _{j}^{2}=0
\]
so $\hat{Z}^{2}=\tilde{Z}^{2}=-2$, and hence $\hat{Z}$ is a positive
root.
\end{proof}

\section{BPS invariants} \label{sec:GVinvarinats}
As mentioned in \S~ \ref{sec:Intro}, BPS invariants of $Y$ are defined as the equivariant residues at the $\C^*$-fixed locus of the corresponding moduli space. More precisely, let $\beta \in H_2(Y,\Z)$. By Theorem \ref{thm:fixed-locus} we already know that the fixed locus is either empty or a single point, and the later happens only when $\beta$ corresponds to a positive root. In fact if $$\beta=\sum_{i=1}^r m_i [F_i]$$ is such a curve class, then we proved this fixed point of moduli space is given by $\O_C$ the structure sheaf of a 1-dimensional \CM scheme having generic multiplicity $m_i$ along $F_i$. We also showed $C$ is moreover scheme-theoretically supported on $S_0$. Now fix such a class $\beta$ and the curve $C$ determined as above. The corresponding BPS invariant is by definition
$$n^L_\beta(Y)=\frac{e(\ext^2(\O_C,\O_C))}{e(\ext^1(\O_C,\O_C))},$$ where $e(-)$ is the equivariant Euler class.

For any nonzero integer $a$ by $\C_a$ we mean the weight $a$
representation of $\C^*$. The canonical bundle of $Y$, is trivial with
weight $-3$. We have $$\hom(\O_C,\O_C) \cong \C \;\;\text{and}\;\;
\ext^3(\O_C,\O_C)\cong \C_{3},$$ where the first isomorphism is
because of the stability of $\O_C$ and the second is because of
equivariant Serre duality. Let $$\chi(\O_C,\O_C)=\sum_{p=0}^3
(-1)^p\ext^p(\O_C,\O_C)$$ be the Euler characteristic where the right
hand side is regarded as a virtual $\C ^{*}$ representation. Since we
already know the first and the last term contributions, if we find
$\chi(\O_C,\O_C)$ we will be able to determine $n^L_\beta(Y)$ (note we
only need the ratio of the Euler classes of the third and second
terms).

By using equivariant Hirzebruch-Riemann-Roch we have (see \cite[Lemma 6.1.3]{Huybrechts-Lehn-book}) \begin{equation} \label{equ:HRR}\CH\left(\chi(\O_C,\O_C)\right)=\int_Y \CH^\vee(\O_C)\CH(\O_C)\td(Y)\end{equation} where $$\CH=\operatorname{ch}_0+\operatorname{ch}_1+\operatorname{ch}_2+\cdots$$ is the equivariant Chern character and $$\CH^\vee=\operatorname{ch}_0-\operatorname{ch}_1+\operatorname{ch}_2-\cdots.$$
$\td(-)$ denotes the equivariant Todd class. By this integral we mean equivariant push-forward from $Y$ to a point.

One way to compute $\CH(\O_C)$ is to find an equivariant locally free
resolution of $\O_C$. However, in most of our cases $C$ is not a local
complete intersection in $Y$ and so it is difficult to find such a
resolution for $\O_C$ in such cases. We use an alternative method to
compute equivariant Chern characters of $\O_C$ at the components of
fixed locus of $Y$. By Atiyah-Bott residue formula, this is all we
need to know in order to evaluate the integral in (\ref{equ:HRR}). As
in the proof of Lemma \ref{lem:euler-char}, we can construct a divisor
$\Ch$ on $S$ invariant under the induced $\C^*$-action and with the
property $f_*\O_{\Ch}=\O_C$ and $R^if_*\O_{\Ch}=0$ for $i>0$. Applying
Grothendieck-Riemann-Roch to the equivariant projective morphism $f:S
\to Y$ we have this relation in the equivariant $K$-group of
$Y$: $$\CH(\O_C)=\frac{1}{\td(Y)}f_*\left(\CH(\O_{\Ch})\cdot
\td(S)\right).$$ Since $\Ch$ is a divisor on a smooth surface, there
is a natural locally free resolution for $\O_{\Ch}$: $$0\to
\O_S(-\Ch)\to \O_S \to \O_{\Ch}\to 0.$$ Let $i_P:P \hookrightarrow Y$
be a $\C^*$-fixed component. By correspondence of residues (e.g. see
\cite[\S~ 3]{Bertram}) we can write
\begin{equation} \label{equ:GRR}
i_P^*(\CH(\O_{C}))=\frac{e(N_{P|Y})}{i_P^*\td(Y)}\cdot\sum_{Q\to
P}j_Q^*f_*\left((1-\CH(\O_S(-\Ch))\cdot
\frac{\td(S)}{e(N_{Q|S})}\right)\end{equation} where $j_Q:Q
\hookrightarrow S$ is a $\C^*$-fixed component in $S$, and the sum is
over all such components mapping into $P$.

Our proof of Theorem~\ref{thm:GVinvariants} is a case by case check
using (~\ref{equ:HRR}) and (~\ref{equ:GRR}). We carry this
out for some of the most complicated roots in the case corresponding
to $E_{8}$; all other cases are similar and simpler. The $E_{8}$ root
system has 120 positive roots and this leads to 36 curve classes on $Y$
having nonzero BPS invariants. The possible values of the nonzero invariants
and the number of classes carrying these invariants and the number of positive roots corresponding to these classes are summarized below:
\begin{center}
\begin{tabular}{|c|c|c|}
\hline
$\sharp$ of positive roots & $\sharp$ of classes $\beta$&$n_{\beta } (Y)$\\ \hline
$32$&$16$&$1$\\ \hline
$4$&$4$&$\frac{1}{2}$\\ \hline
$48$&$12$&$2$\\ \hline
$32$&$4$&$4$\\ \hline
\end{tabular}
\end{center}
Note that the entries of the first column add up to 116, because there are 4 positive simple roots that do not correspond to a curve class on $Y$. These simple roots correspond to black vertices of the $E_8$ Dynkin diagram in Figure \ref{fig:ADE}.

In the following proposition we list the largest class $\beta$ carrying
each possible nonzero invariant.

\begin{prop}
Let $G=A_5$ be the alternating group in 5 elements; this case corresponds to the root system $E_8$. Let $F_1, \dots ,F_4$ be the components of $F$ as in Figure \ref{fig:E8}. Then $$\CH\left(\chi(\O_C,\O_C)\right)=\begin{cases}1-\e^{3t} &  C=3F_1+5F_2+4F_3+3F_4 \\
1-\e^{2t}+\e^{t}-\e^{3t}  &   C=2F_1+4F_2+4F_3+2F_4 \\
1-\e^{t}+\e^{2t}-\e^{3t}  &   C=2F_1+4F_2+3F_3+2F_4 \\
1-2\e^{t}+2\e^{2t}-\e^{3t} &  C=\,\,\,F_1+2F_2+2F_3+\,\,\,F_4
\end{cases}$$ where $t$ is the equivariant parameter.
\end{prop}
\begin{figure}
\centering
\includegraphics{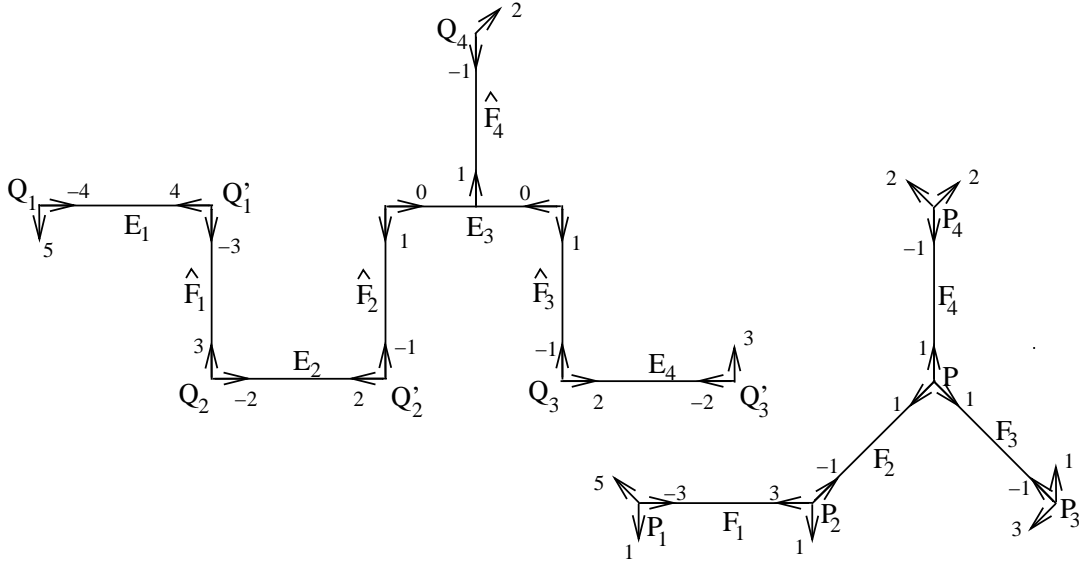}
\caption{The fibers over origin on $\hat{A}_5\hilb(\C^2)$ (left) and in $A_5\hilb(\C^3)$ (right)}
\label{fig:E8}
\end{figure}

The following corollary is now immediate:
\begin{cor}
Let $G$ and $F_1, \dots, F_4$ be as in proposition and let $\beta =k_{1}[F_{1}]+k_{2}[F_{2}]+k_{3}[F_{3}]+k_{4}[F_{4}]$. Then
\begin{center}
\begin{tabular}{|c|c|c|c|}
\hline
$(k_{1},k_{2},k_{3},k_{4}) $&$e (\ext ^{2} (\O _{C},\O _{C})) $&$e (\ext ^{1} (\O _{C},\O _{C})) $& $n_{\beta } (Y)$\\ \hline
$(3,5,4,3) $&$1$&$1$&$1$\\ \hline
$(2,4,4,2) $&$t$&$2t$&$\frac{1}{2}$\\ \hline
$(2,4,3,2) $&$2t$&$t$&$2$\\ \hline
$(1,2,2,1) $&$(2t)^{2}$&$t^{2}$&$4$\\ \hline
\end{tabular}
\end{center}
The values in the last column are as given in Theorem
\ref{thm:GVinvariants}. Note that the values given in the second and
the third columns are modulo possibly cancelling factors common in two
Euler classes.
\end{cor}

\begin{proof}[Proof of Proposition.] In Figure \ref{fig:E8}, $P_1,\,P_2,\,P_3,\,P_4,\,P$ and $$Q_1,\,Q'_1,\,Q_2,\,Q'_2,\,Q_3,\,Q'_3,\,Q_4,\,E_3$$ are the $\C^*$-fixed components of $$Y=A_5\hilb(\C^3)\;\; \text{and} \;\; S=\hat{A}_5\hilb(\C^2),$$ respectively. $f$ maps $Q_i$ and $Q'_i$ to $P_i$ and $E_3$ to $P$. The numbers in this figure stand for $\C^*$-weights of the tangent spaces at the fixed points. We only prove the first identity; the proof of the rest is similar.
In this case $$\Ch=3\Fh_1+5\Fh_2+4\Fh_3+3\Fh_4+2E_1+4E_2+6E_3+2E_4.$$ Note that this corresponds to the longest root in the root system $E_8$. Using (\ref{equ:HRR}) and (\ref{equ:GRR}), we compute the contribution of each $\C^*$-fixed component of $Y$ to $\CH\left(\chi(\O_C,\O_C)\right)$. We denote by $\chi_i$ the contribution of $P_i$, and by $\chi$ the contribution of $P$. For simplicity, we let $\mu=\e^{t}$, and $\CH(P)=i_P^*\CH(\O_C)$ and $\CH(Q)=j_Q^*\CH(\O_{\Ch})$ and similarly for $P_i$'s, $Q_i$'s and $E_3$. We use the same notation for $\CH^\vee$. Finally by $\CH(Q_i\to P_i)$ we mean the contribution of $j_{Q_i}^*\CH(\O_{\Ch})$ to $i_{P_i}^*\CH(\O_C)$ in (\ref{equ:GRR}), and similarly for $\CH(Q'_i \to P_i)$.
\begin{itemize}
\item \emph{Contribution of $P_1$:} We have $\CH(Q_1)=1-\mu^{-10}$, $$\frac{i^*_{P_1}\td(Y)}{e(N_{P_1|Y})}=  \frac{1}{(1-\mu^{-5})(1-\mu^{-1})(1-\mu^3)}  \text{\;\;\;\; and}$$
    $$\frac{j^*_{Q_1}\td(S)}{e(N_{Q_1|S})}=  \frac{1}{(1-\mu^{-5})(1-\mu^{4})}. $$ Hence $\displaystyle \CH(Q_1\to P_1)=\frac{(1-\mu^{-10})(1-\mu^{3})(1-\mu^{-1})}{(1-\mu^4)}$ and similarly $$\CH(Q'_1\to P_1)=\frac{(1-\mu^{-6})(1-\mu^{-5})(1-\mu^{-1})}{(1-\mu^{-4})}.$$ In the last formula we used $\CH(Q'_1)=1-\mu^{-6}$. So we have \begin{align*}
    \CH(P_1)&=\CH(Q_1\to P_1)+\CH(Q'_1\to P_1)\\&={\mu}^{-11}-{\mu}^{-10}+{\mu}^{-7}-{\mu}^{-6}-{\mu}^{-2}+1
    \end{align*} and hence $\CH^\vee(P_1)=1-{\mu}^{2}-{\mu}^{6}+{\mu}^{7}-{\mu}^{10}+{\mu}^{11}.$ We finally get the contribution of $P_1$ to (\ref{equ:HRR}):
    \begin{align*}\chi_1&=\frac{\CH^\vee(P_1)\cdot \CH(P_1)\cdot i^*_{P_1}td(Y)}{e(N_{P_1|Y})}\\&={\mu}^{-5}-{\mu}^{-3}+{\mu}^{-2}+{\mu}^{-1}-\mu+{\mu}^{2}-{\mu}^{4}-{
\mu}^{5}+{\mu}^{6}-{\mu}^{8}.
 \end{align*}
\item \emph{Contributions of $P_2,\,P_3,\,P_4$:} These cases are quite similar to the case of $P_1$; we only summarize the results below:

$\displaystyle \CH(Q_2 \to P_2)=\frac{(1-\mu^{-6})(1-\mu^{-1})(1-\mu)(1-\mu^{-3})}{(1-\mu^2)(1-\mu^{-3})},$

$\displaystyle \CH(Q'_2 \to P_2)=\frac{(1-\mu^{-6})(1-\mu^{-1})(1-\mu)(1-\mu^{-3})}{(1-\mu^{-2})(1-\mu)},$

$\displaystyle \CH(Q_3 \to P_3)=\frac{(1-\mu^{-6})(1-\mu^{-1})(1-\mu)(1-\mu^{-3})}{(1-\mu^{-2})(1-\mu)},$

$\displaystyle \CH(Q'_3 \to P_3)=\frac{(1-\mu^{-6})(1-\mu)(1-\mu^{-1})(1-\mu^{-3})}{(1-\mu^{2})(1-\mu^{-3})},$

$\displaystyle \CH(Q_4 \to P_4)=(1-\mu^{-6})(1-\mu^{-2}),$

$\CH(P_2)={\mu}^{-8}-{\mu}^{-6}-{\mu}^{-2}+1,$

$\CH(P_3)={\mu}^{-8}-{\mu}^{-6}-{\mu}^{-2}+1,$

$\CH(P_4)={\mu}^{-8}-{\mu}^{-6}-{\mu}^{-2}+1,$

$\chi_2={\mu}^{-4}+2\,{\mu}^{-3}+{\mu}^{-2}+{\mu}^{-1}+2+\mu-{\mu}^{2}-2\,{
\mu}^{3}-{\mu}^{4}-{\mu}^{5}-2\,{\mu}^{6}-{\mu}^{7},$

$\chi_3={\mu}^{-4}+2\,{\mu}^{-3}+{\mu}^{-2}+{\mu}^{-1}+2+\mu-{\mu}^{2}-2\,{
\mu}^{3}-{\mu}^{4}-{\mu}^{5}-2\,{\mu}^{6}-{\mu}^{7},$

$\chi_4={\mu}^{-4}+{\mu}^{-3}+{\mu}^{-2}+{\mu}^{-1}+1+\mu-{\mu}^{2}-{\mu}^{3}
-{\mu}^{4}-{\mu}^{5}-{\mu}^{6}-{\mu}^{7}.
$
\item \emph{Contribution of $P$:}  Note that $-\Ch\cdot E_3=0$, and hence $\O_S(-\Ch)|_{E_3}$ is trivial with weight -6, so $\CH(E_3)=1-\mu^{-6}$. One sees easily that $$\frac{i^*_{P}\td(Y)}{e(N_{P|Y})}= \frac{1}{(1-\mu ^{-1})^3}.$$ Let $\x$ be the point class on $E_3\cong \P^1$ then $$\frac{j^*_{E_3}\td(S)}{e(N_{E_3|S})}=\frac{2\x}{(1-\e^{-t+2\x})(1-\e^{-2\x})}=\left(\frac{1}{1-\mu^{-1}}-\frac{2\x \mu^{-1}}{(1-\mu^{-1})^2}\right)(1+\x).$$ So we get
    \begin{align*}
    \CH(P)&=(1-\mu^{-6})(1-u^{-1})^3 \cdot f_*\left(\frac{(1+\x)}{1-\mu^{-1}}-\frac{2\x \mu^{-1}}{(1-\mu^{-1})^2}\right)\\&=\mu^{-8}-\mu^{-6}-\mu^{-2}+1.
    \end{align*}
    And finally we get
    \begin{align*}
    \chi&=\frac{(1-\mu^{2}-\mu^{6}+\mu^8)(\mu^{-8}-\mu^{-6}-\mu^{-2}+1)}{(1-\mu^{-1})^3}\\&=-{\mu}^{-5}-3\,{\mu}^{-4}-4\,{\mu}^{-3}-4\,{\mu}^{-2}-4\,{\mu}^{-1}-4
-2\,\mu+2\,{\mu}^{2}+4\,{\mu}^{3}\\&+4\,{\mu}^{4}+4\,{\mu}^{5}+4\,{\mu}^{
6}+3\,{\mu}^{7}+{\mu}^{8}.
\end{align*}
  \end{itemize}
At the end we add all the contributions $$\CH\left(\chi(\O_C,\O_C)\right)=\chi+\chi_1+\chi_2+\chi_3+\chi_4=1-\mu^3.$$
\end{proof}

\bibliography{../bibtex/mainbiblio}

\begin{thebibliography}{10}

\bibitem{Artin}
Michael Artin.
\newblock On isolated rational singularities of surfaces.
\newblock {\em Amer. J. Math.}, 88:129--136, 1966.

\bibitem{Bertram}
Aaron Bertram.
\newblock Another way to enumerate rational curves with torus actions.
\newblock {\em Invent. Math.}, 142(3):487--512, 2000.
\newblock arXiv:math/9905159.

\bibitem{Boissiere-Sarti}
Samuel Boissi{\`e}re and Alessandra Sarti.
\newblock Contraction of excess fibres between the {M}c{K}ay correspondences in
  dimensions two and three.
\newblock {\em Ann. Inst. Fourier (Grenoble)}, 57(6):1839--1861, 2007.
\newblock math.AG/math/0504360.

\bibitem{BGh-Ghilb}
Jim Bryan and Amin Gholampour.
\newblock {The Quantum McKay Correspondence for polyhedral singularities }.
\newblock arXiv:0803.3766.

\bibitem{Nak-Ghilb}
Yasushi Gomi, Iku Nakamura, and Ken-ichi Shinoda.
\newblock Hilbert schemes of {$G$}-orbits in dimension three.
\newblock {\em Asian J. Math.}, 4(1):51--70, 2000.
\newblock Kodaira's issue.

\bibitem{Go-Va}
Rajesh Gopakumar and Cumrun Vafa.
\newblock M-theory and topological strings--{II}, 1998.
\newblock Preprint, hep-th/9812127.

\bibitem{Huybrechts-Lehn-book}
Daniel Huybrechts and Manfred Lehn.
\newblock {\em The geometry of moduli spaces of sheaves}.
\newblock Friedr. Vieweg \& Sohn, Braunschweig, 1997.

\bibitem{Katz-BPS}
Sheldon Katz.
\newblock Genus zero {G}opakumar-{V}afa invariants of contractible curves.
\newblock {\em J. Differential Geom.}, 79(2):185--195, 2008.
\newblock arXiv:math/0601193.

\bibitem{PT-BPS}
Richard~Thomas Rahul~Pandharipande.
\newblock {Stable pairs and BPS invariants }.
\newblock arXiv:0711.3899.

\bibitem{Thomas-Casson}
R.~P. Thomas.
\newblock A holomorphic {C}asson invariant for {C}alabi-{Y}au 3-folds, and
  bundles on {$K3$} fibrations.
\newblock {\em J. Differential Geom.}, 54(2):367--438, 2000.
\newblock arXiv:math/9806111.

\end{thebibliography}
\bibliographystyle{plain}

\noindent Department of Mathematics, University of British Columbia\\
\medskip
\noindent Department of Mathematics, California Institute of Technology\\
\medskip
\noindent Email address: jbryan@math.ubc.ca, agholamp@caltech.edu

 \end{document}